\documentclass{amsart}
\usepackage{amsfonts}
\usepackage{amsmath,amssymb}
\usepackage{amsthm}
\usepackage{amscd}
\usepackage{graphics}
\usepackage{graphicx}
\theoremstyle{definition}{
	\newtheorem{Def}{{\rm Definition}}
	\newtheorem{Ex}{{\rm Example}}
	
	\newtheorem{Prob}{{\rm Problem}}
}
\theoremstyle{plain}
{
	
	\newtheorem{Prop}{Proposition}
	\newtheorem{Thm}{Theorem}
	\newtheorem{MainThm}{Main Theorem}

}

\begin{document}
	\title[Decomposing manifolds into submanifolds via specific fold maps]{Decompositions of manifolds into submanifolds compatible with specific fold maps}
	\author{Naoki Kitazawa}
	\address{Institute of Mathematics for Industry, Kyushu University, 744 Motooka, Nishi-ku Fukuoka 819-0395, Japan\\
		TEL (Office): +81-92-802-4402 \\
		FAX (Office): +81-92-802-4405 \\
	}
	\email{n-kitazawa@imi.kyushu-u.ac.jp}
	\urladdr{https://naokikitazawa.github.io/NaokiKitazawa.html}
	
	\begin{abstract}
		We present new explicit decompositions of manifolds via so-called {\it fold} maps into lower dimensional spaces. Fold maps form a nice class of so-called {\it generic} maps, generalizing Morse functions naturally.
		
	    To understand the topologies and the differentibale structures of manifolds globally, decomposing manifolds are important and this presents interesting topics and problems on geometry of manifolds. The notion of a {\it Heegaard splitting} of a $3$-dimensional closed and connected manifold presents a pioneering study. A $3$-dimensional closed and connected manifold is always decomposed into two copies of a so-called $3$-dimensional {\it handlebody} via a so-called {\it Heegaard surface}, which is a closed and connected surface.
	     Heegaard splitiings are generalized as {\it multisections} of smooth or PL manifolds in the 2010s. As a way of understanding, these decompositions are understood via Morse functions and general generic smooth maps whose codimensions are negative.


		
	\end{abstract}
	
	
	\maketitle
	\section{Introduction.}
	\label{sec:1}
	
	Decomposing manifolds suitably gives natural and important tools in understanding global algebraic topological or differential topological properties of manifolds. A $3$-dimensional closed and connected manifold is always decomposed into so-called two copies of a so-called $3$-dimensional {\it $1$-handlebody} via a closed and connected surface, which is a so-called {\it Heegaard surface}. For related theory on $3$-dimensional manifolds, see \cite{hempel}.
	
	Such decompositions are generalized to ({\it generalized}) {\it multisections} of smooth or PL manifolds, for example. \cite{gaykirby} is a pioneering study for $4$-dimensional cases. \cite{rubinsteintillmann1,rubinsteintillmann2} study generalized versions, introducing the notion of a ({\it generalized}) {\it multisection}. \cite{lambertcolemiller} studies the $5$-dimensional case explicitly related to these studies.
	 
		They are in some situations studied and understood via Morse functions and more general so-called {\it generic} smooth maps.  "{\it Generic}" means general in the situation we discuss problems. Morse functions are simplest generic maps. Fundamental or some advanced theory of generic smooth maps is systematically presented in \cite{golubitskyguillemin} for example. 
		
		In Heegaard splittings, classically, Morse functions are fundamental tools. There always exists a Morse function compatible with a Heegaard splitting: a so-called {\it twisted double} of a Morse function the preimage of whose minimum or maximum is a Heegaard surface.
		For studies via generic smooth maps into the plane, see \cite{johnson, kobayashisaeki, takao} for example. 
		As related to classifications of Heegaard splittings, they mainly concentrate on numerical difference or some distance of two Heegaard splittings. They consider the product maps of two Morse functions compatible with the given two Heegaard splittings. 
		For the $4$-dimensional case, \cite{gaykirby} uses generic smooth maps into the plane. They call the maps {\it $2$-Morse functions}. \cite{baykursaeki1,baykursaeki2} are also related pioneering studies. These studies deform generic smooth maps into the plane by several local operations and obtain generic ones compatible with the decompositions. 
		
		It is natural and difficult to find nice explicit classes of such decompositions, especially in the case where the dimensions of the manifolds are higher. It is also natural and challenging to construct nice examples for the classes. This is closely related to difficulty of construction of explicit smooth maps and manifolds whereas knowing the existence is in some senses not so difficult as \cite{eliashberg1,eliashberg2} show by Eliashberg's celebrating existence theory of {\it fold} maps using methods from some theory of existence of solutions to differential equations. {\it Fold} maps are higher dimensional versions of Morse functions and presented later.
		
		Our paper concentrates on the following related problems. Some answers are in \cite{kitazawa9} and we give additional answers. 
		
		\begin{Prob}
			\label{prob:1}
		Can we define nice classes of decompositions. Especially, by using suitable classes of smooth maps whose codimensions are negative.
		\end{Prob}

    	\begin{Prob}
    		\label{prob:2}
    	Can we construct nice examples belonging to classes of Problem \ref{prob:1}.
    \end{Prob}
    Our study considers {\it fold} maps, which are simplest higher dimensional generalizations of Morse functions and which are locally represented as projections or the product maps of Morse functions and the identity maps on disks. This is also generic.
    
     As an explicit class, we mainly consider {\it special generic} maps, related pioneering studies on which are in \cite{burletderham,furuyaporto,saeki1} for example. This class contains Morse functions on spheres, playing important roles in Reeb's theorem, and canonical projections of spheres naturally embedded in Euclidean spaces, or {\it unit spheres}. This class is considered in \cite{kitazawa8} and we have given several answers to Problems \ref{prob:1} and \ref{prob:2} there. There we introduce {\it smooth} or {\it PL} {\it near multisections} of smooth or PL manifolds as a natural class of decompositions and give several examples.
      We give additional answer, as Main Theorem \ref{mthm:1} states in the following. We leave rigorous definitions of several undefined notions later. 
      Here we only note some. A {\it unit disk} is the disk bounded by the unit sphere in the Euclidean space we consider. A {\it linear} bundle whose fiber is a unit sphere means a bundle whose structure group consists of linear transformations. It is {\it reduced by degree $k^{\prime}$} if the degree of the structure group is reduced to a group of linear transformations whose degree drops by $k^{\prime}>0$ from the so-called {\it orthgonal group} whose {\it degree} is the sum of the dimension of the sphere of the fiber and $1$. 
     {\it Smooth} or {\it PL} {\it admissible} decompositions and as specific cases, {\it GN}, {\it GANM}, and {\it DGNANM} {\it decompositions} are defined rigorously first    
      in our paper as nice decompositions of smooth or PL manifolds. {\it Disk decompositions} are decompositions into copies of unit disks and also multisections.
      In some of studies related to multisections and similar decompositions, decompositions belonging to the classes appear explicitly.     
    	\begin{MainThm}[Theorem \ref{thm:6}]
    		\label{mthm:1}
    		Let $M_0$ be a smooth closed and connected manifold of dimension $m_0>0$ admitting an {\rm admissible smooth (PL) decomposition} of {\rm degree $k$} {\rm defined by a family $\{M_{0,j}\}_{j=1}^{k_0}$ of submanifolds}. 
    Let $M$ be the total space of a linear bundle over $M_0$ whose fiber is the $k$-dimensional unit sphere and which is reduced by degree $k^{\prime}$ with integers $k>0$ and $k^{\prime}>0$ satisfying $k-k^{\prime}>0$.
    
    Then, $M$ admits an admissible smooth {\rm (}resp. PL{\rm )} decomposition of degree $2k_0$.
    In the case where the given decomposition is a {\it GN} decomposition, the resulting decomposition is {\it GANM}. In the case where the given decomposition is {\it DGNANM}, so is the resulting one. In the case where the given decomposition is a multisection with the structure group of the given linear bundle being reduced to a so-called {\rm rotation group}, the resulting decomposition is also a multisection. In the case where the given decomposition is a{\it disk decomposition}, so is the resulting decomposition.
    		
    \end{MainThm}

     Note that the image of a special generic map on a closed manifold is a compact manifold smoothly immersed into the Euclidean space via codimension $0$ immersion. We also present closely related other new results and observations.
     

     We also consider {\it round} fold maps, introduced by the author in \cite{kitazawa0.1,kitazawa0.2,kitazawa0.3} and studied in \cite{kitazawa0.4,kitazawasaeki1} and also in \cite{kitazawa0.5,kitazawasaeki2}. Using this class is a new work in our paper. A {\it round} fold map is a Morse function obtained by gluing two copies of a Morse function on a compact manifold in a natural way as in Heegaard splittings or a fold map the image of whose singular set is embedded concentrically. Main Theorem \ref{mthm:2} gives a part of our new results. A round fold map {\it having a globally trivial monodormy} means a round fold map which is a function or one obtained by gluing the product map of a Morse function on a compact manifold, called an {\it axis manifold} and the identity map on a unit sphere and the projection on the product of the unit disk and a closed smooth manifold onto the unit disk. Such undefined terminologies and notions are rigorously presented later.
		\begin{MainThm}[Some of Theorem \ref{thm:7}]
		\label{mthm:2}
		Let $m>n \geq 1$ be integers. Let $M_L$ be an {\rm (}$m-n+1${\rm )}-dimensional smooth, compact and connected manifold.
	Assume that there exists an admissible smooth {\rm (}PL{\rm )} decomposition defined by the submanifolds $\{M_{L,j}\}_{j=1}^k$. Then a manifold $M$ admitting a round fold map into ${\mathbb{R}}^n$ having a globally trivial monodromy and an axis manifold diffeomorphic {\rm (}resp. PL homeomorphic{\rm )} to $M_L$ admits an admissible smooth {\rm (}resp. PL{\rm )} decomposition of degree $2k$. In addition in the case where the given admissible decomposition defined by the submanifolds $\{M_{L,j}\}_{j=1}^k$ is a disk decomposition, the resulting admissible decomposition of degree $2k$ is also a disk decomposition. In the case where the given admissible decomposition defined by the submanifolds $\{M_{L,j}\}_{j=1}^k$ is a multisection, the resulting admissible decomposition of degree $2k$ is a multisection.
	In addition, in the case where the given admissible decomposition is DGNANM, it is also DGNANM {\rm outside an} {\rm (}$m-n${\rm )}{\rm-dimensional smooth} {\rm (}resp. {\rm PL} {\rm )} {\rm compact submanifold with no boundary}.
	\end{MainThm}
	Other than Main Theorems, we present related new results and observations such as Theorem \ref{thm:8} and Example \ref{ex:6}. The organization is as follows. The second section reviews important facts on special generic maps and round fold maps. The third section is for decompositions of manifolds and several important classes of them. Some are presented in \cite{kitazawa9} and some are new and a part of our new results. The fourth section is for the so-called {\it Lusternik-Schnirelmann category number} of a topological space, systematically presented in \cite{cornealuptonopreatanre}. It is a homotopy invariant for topological spaces, defined as the minimal number among the numbers of subspaces whose inclusions into the space are null-homotopic and which cover the space. Decompositions of disks give such covers whereas they may not give the minimal numbers. This is closely related to some of our Main Theorems. The fifth section is for Main Theorems. The sixth section presents conclusion with future problems.\\
	\ \\
	{\bf Conflict of Interest.} \\
	The author is a member of the project JSPS KAKENHI Grant Number JP22K18267 "Visualizing twists in data through monodromy" (Principal Investigator: Osamu Saeki). The present study is due to this project. \\
	\ \\
	{\bf Data availability.} \\
	Data supporting our present study essentially are all included in our paper.
	
	\section{Fundamental properties and existing studies on special generic maps and the manifolds.}
	The $k$-dimensional Eucliedan space, which is a simplest smooth manifold and a Riemannian manifold equipped with the so-called standard Euclidean metric, is denoted by ${\mathbb{R}}^k$ for $k \geq 1$. Put $\mathbb{R}:={\mathbb{R}}^1$.
    For $x \in {\mathbb{R}}^k$ $||x|| \geq 0$ denotes the distance to the origin $0 \in {\mathbb{R}}^k$ under the metric induced by the standard Euclidean metric. $S^k:=\{x \in {\mathbb{R}}^{k+1} \mid ||x||=1 \}$ is the {\it $k$-dimensional unit sphere} for $k \geq 0$. It is a $k$-dimensional smooth closed submanifold of ${\mathbb{R}}^{k+1}$ and has no boundary. $S^k$ is connected for $k \geq 1$ and the two-point set which is discrete for $k=0$. $D^k:=\{x \in {\mathbb{R}}^{k} \mid ||x|| \leq 1 \}$ is the {\it $k$-dimensional unit disk} for $k \geq 1$. It is a $k$-dimensional smooth compact and connected submanifold of ${\mathbb{R}}^{k+1}$ whose boundary is the unit sphere $S^{k-1}$.

	For a topological space $X$ homeomorphic to a cell complex whose maximal dimension is finite, $\dim X$ denotes the dimension, which is uniquely defined.
	
	(Topological) manifolds are known to have the structures of CW complexes. Every smooth manifold is canonically regarded as a polyhedron. This is called a {\it PL manifold}. This canonical PL structure is uniquely defined.
	
	It is also well-known that a topological manifold whose dimension is at most $3$ has a PL structure uniquely. Note that this does not have PL structures which are not the PL manifolds. Furthermore, the manifold is uniquely regarded as a smooth manifold. This is on so-called hauptvermutung. See \cite{moise} for example.
	
	For general expositions on the PL category and the piecewise smooth category, which are known to be equivalent category, especially on PL topological theory, see \cite{hudson} for example. We omit rigorous expositions on fundamental notions, terminologies and theorems on the theory.
	
	For a differentiable map $c:X \rightarrow Y$ between differentiable manifolds, $p \in X$ is a {\it singular} point of $c$ if the rank of the differential ${dc}_p$ at $p$ is smaller than $\min \{\dim X,\dim Y\}$. $c(p)$ is called a {\it singular value} of $c$. $S(c)$ denotes the set of all singular points of $c$ and is called the {\it singular set} of $c$.
	
	A {\it diffeomorphism} between smooth manifolds means a smooth map with no singular points being also a homeomorphism. A {\it diffeomorphism on a smooth manifold} means a diffeomorphism from the manifold onto itself. Two manifolds are {\it diffeomorphic} if and only if there exists a diffeomorphism between the two manifolds. This of course gives an equivalence relation on the class of all smooth manifolds with the corners being eliminated. We can define {\it PL homeomorphic manifolds} via piecewise smooth homeomorphisms similarly.
	 
	In our paper, we use the notion of a corner of a manifold and we can eliminate or create corners depending on the situations. It is well-known that smooth manifolds with corners are smoothed by eliminating the corners in a canonical way. It is also well-known that any pair of such manifolds obtained from a fixed smooth manifold in this way is always diffeomorphic.  
	\begin{Def}
		\label{def:1}
		A smooth map $c:X \rightarrow Y$ between two smooth manifolds with no boundaries are said to be a {\it fold} map if at each singular point $p \in X$ there exists suitable local coordinates around $p$ and an integer $0 \leq i(p) \leq \frac{\dim X-\dim Y+1}{2}$ and $c$ is represented as $(x_1,\cdots,x_{\dim X}) \rightarrow (x_1,\cdots,x_{\dim Y-1},{\Sigma}_{j=1}^{\dim X-\dim Y-i(p)+1} {x_{\dim Y+j-1}}^2-{\Sigma}_{j=1}^{i(p)} {x_{\dim X-i(p)+j}}^2)$.
	\end{Def}
	It is a kind of fundamental exercises on smooth manifolds, Morse functions, and singularity theory of smooth maps, to see that the canonical projection of a unit sphere, defined as a map mapping $(x_1,x_2) \in S^k \subset {\mathbb{R}}^{k+1}={\mathbb{R}}^{k_1} \times {\mathbb{R}}^{k_2}$ into $x_1 \in {\mathbb{R}}^{k_1}$ where $k \geq 2$, $k_1,k_2 \geq 1$ and $k=k_1+k_2$. 
	Since the 1990s, manifolds admitting special generic maps have been studied by Saeki and Sakuma as \cite{saeki1,saeki2,saekisakuma1,saekisakuma2}, followed by \cite{nishioka,wrazidlo1,wrazidlo2,wrazidlo3}.
	They have revealed restrictions on the differentiable structures of the homotopy spheres etc. and the homology groups. As a pioneer, the author has studied the cohomology rings of the manifolds mainly in \cite{kitazawa1,kitazawa2,kitazawa3,kitazawa4,kitazawa5,kitazawa6,kitazawa7,kitazawa9} for example.

	The {\it diffeomorphism group} of a smooth manifold is the group of all diffeomorphisms on it topologized with the {\it Whitney $C^{\infty}$ topology}. This topology is one of fundamental topologies of spaces of smooth maps between smooth manifolds. See \cite{golubitskyguillemin}.
	
	A {\it smooth} bundle means a bundle with a fiber being a smooth manifold and the structure group defined as the diffeomorphism group of the fiber. A {\it linear} bundle means a bundle whose fiber is a Euclidean space, unit sphere, or a unit disk and whose structure group consists of linear transformations where linear transformations are defined in a natural and canonical way. Of course linear bundles are smooth. For general theory of bundles, consult \cite{steenrod}. Consult also \cite{milnorstasheff} for linear bundles or so-called {\it vector bundles}, which are specific linear bundles and whose fibers are vector spaces.
	The so-called {\it orthogonal group} of {\it degree $k$} is regarded as a natural structure group of a linear bundle whose fiber is the Euclidean space ${\mathbb{R}}^k$, the unit disk $D^k$, and the unit sphere $S^{k-1}$ for $k \geq 1$.
	For a linear bundle whose fiber is as this with a positive integer $k^{\prime}>0$ satisfying $k-k^{\prime}>0$, it is said to be {\it reduced by degree $k^{\prime}$} if the structure group is reduced to a natural Lie subgroup isomorphic to the orthogonal group of degree $k-k^{\prime}$.
	\begin{Prop}[\cite{saeki1,saeki2}]
		\label{prop:1}
		A special generic map $f:M \rightarrow N$ on an $m$-dimensional closed and connected manifold $M$ into an $n$-dimensional connected and non-compact manifold $N$ with no boundary enjoys the following properties.
		\begin{enumerate}
			\item \label{prop:1.1}
			There exists an $n$-dimensional compact and connected smooth manifold $W_f$, a smooth surjection $q_f:M \rightarrow W_f$ and a smooth immersion $\bar{f}:W_f \rightarrow N$ enjoying the relation $f=\bar{f} \circ q_f$. Furthermore $q_f$ can be taken as a map mapping the singular set $S(f)$ of $f$ onto the boundary $\partial W_f \subset W_f$ as a diffeomorphism.
			\item \label{prop:1.2}
			There exists a small collar neighborhood $N(\partial W_f)$ of the boundary $\partial W_f \subset W_f$. Furthermore, we can have one so that the following two are enjoyed.
			\begin{enumerate}
				\item \label{prop:1.2.1}
				The composition of the restriction of $q_f$ to the preimage ${q_f}^{-1}(N(\partial W_f))$ with the canonical projection to $\partial W_f$ gives a linear bundle whose fiber is the {\rm (}$m-n+1${\rm )}-dimensional unit disk $D^{m-n+1}$.
				\item \label{prop:1.2.2}
				The restriction of $q_f$ to the preimage of $W_f-{\rm Int}\ N(\partial W_f)$ gives a smooth bundle whose fiber is the {\rm (}$m-n${\rm )}-dimensional unit sphere $S^{m-n}$. In some specific case, such as the case $m-n=0,1,2,3$, the bundle is linear.
		\end{enumerate}
	\end{enumerate}
	\end{Prop}
\begin{Def}[\cite{kitazawa9}]
	\label{def:2}
	For a special generic map $f$ in Proposition \ref{prop:1}, the bundle of (\ref{prop:1.2.1}) is called the {\it boundary linear bundle} of $f$ and that of (\ref{prop:1.2.2}) is called the {\it internal smooth bundle} of $f$.
\end{Def}

The following gives a necessary and sufficent condition for a closed and connected manifold to admit a special generic map into a fixed connected and non-compact manifold with no boundary.

	\begin{Prop}[\cite{saeki1}]
		\label{prop:2}
		Let $\bar{N}$ be an $n$-dimensional smooth, compact and connected manifold.
		Given a smooth immersion ${\bar{f}}_N:\bar{N} \rightarrow {\mathbb{R}}^n$.
			
		Then we have a special generic map $f:M \rightarrow {\mathbb{R}}^n$ on a suitable $m$-dimensional closed and connected manifold $M$ enjoying the properties {\rm (}\ref{prop:1.1}{\rm )}, {\rm (}\ref{prop:1.2.1}{\rm )} and {\rm (}\ref{prop:1.2.2}{\rm )} of Proposition \ref{prop:1}.

		Furthermore, let us abuse the notation, and assume the existence of the following bundles where $N(\partial \bar{N})$ is a small collar neighborhood of the boundary $\partial \bar{N}$.
 		 \begin{itemize}
		 	\item A smooth bundle $S_{W_f,S^{m-n}}$ over $\bar{N}-{\rm Int}\ N(\partial \bar{N})$ whose fiber is the unit sphere $S^{m-n}$ and whose restriction to the boundary $\partial (\bar{N}-{\rm Int}\ N(\bar{N}))$ is regarded as a linear bundle.
		 	\item A linear bundle $S_{\partial W_f,D^{m-n+1}}$ over $\partial (\bar{N}-{\rm Int}\ N(\bar{N}))$ whose fiber is the unit disk $D^{m-n+1}$ and whose subbundle obtained by restricting the fiber to the boundary is a linear bundle equivalent to the previous linear bundle over $\partial (\bar{N}-{\rm Int}\ N(\partial \bar{N}))$ as a linear bundle.
		 	\end{itemize}
	We can have a special generic map $f$ enjoying the following properties where we consider suitable identifications between $\bar{N}$ and $W_f$ and their boundaries and natural identifications between the chosen small collar neighborhoods $N(\partial \bar{N})$ and $N(\partial W_f)$ and $\partial (\bar{N}-{\rm Int}\ N(\partial \bar{N}))$ and $\partial (W_f-{\rm Int}\ N(\partial W_f))$.

		\begin{enumerate}
		
			\item \label{prop:2.1}
			The internal smooth bundle of $f$ is a smooth bundle equivalent to the given smooth bundle $S_{W_f,S^{m-n}}$ over $W_f-{\rm Int} N(\partial W_f)=\bar{N}-{\rm Int}\ N(\partial \bar{N})$.
				\item \label{prop:2.2}
			The boundary linear bundle of $f$ is equivalent to the given linear bundle $S_{\partial W_f,D^{m-n+1}}$.
			
		\end{enumerate}
	Moreover, we can construct a map and an $m$-dimensional closed and connected manifold by gluing the two bundles before via some bundle isomorphism between the smooth bundles defined canonically on the boundaries. In addition, conversely, a special generic map always has such a structure. 
	\end{Prop}
	Other than canonical projections of unit spheres, we explain about simplest special generic maps.
	\begin{Ex}
		\label{ex:1}
		Let $l>0$ be an arbitrary positive integer and $m \geq n \geq 2$ integers.
		We choose an arbitrary integer $1 \leq n_j \leq n-1$ for each integer $1 \leq j \leq l$. We consider a connected sum of $l>0$ manifolds the $j$-th manifold of which is the $j$-th manifold in the sequence $\{S^{n_j} \times S^{m-n_j}\}_{j=1}^l$ in the smooth category. We have a special generic map $f:M \rightarrow {\mathbb{R}}^n$ on the resulting manifold $M$ such that in Proposition \ref{prop:2}, the map $\bar{f}$ is an embedding and that its internal smooth bundle and boundary linear bundle are trivial.
		
	\end{Ex}
	
	We introduce several results on special generic maps and manifolds admitting such maps.
	A {\it homotopy sphere} means a smooth manifold which is homeomorphic to a unit sphere whose dimension is at least 1. 
	A {\it standard} (an {\it exotic}) sphere is a homotopy sphere which is diffeomorphic to some unit sphere (resp. not diffeomorphic to any unit sphere). Except $4$-dimensional exotic spheres, homotopy spheres are mutually PL homeomorphic as the canonically defined PL manifolds.
	It is also well-known that $4$-dimensional exotic spheres do not enjoy this property.
	
	\begin{Thm}[\cite{saeki1,saeki2}]
		\label{thm:1}
		\begin{enumerate}
			\item
			\label{thm:1.1}
			Let $m$ be an arbitrary integer satisfying $m \geq 2$.
			An $m$-dimensional closed and connected manifold $M$ admits a special generic map $f:M \rightarrow {\mathbb{R}}^2$ if and only if $M$ is either of the following manifolds.
			\begin{enumerate}
				\item A homotopy sphere whose dimension is not $4$.
				\item A $4$-dimensional standard sphere.
				\item A manifold
				represented as a connected sum of smooth manifolds taken in the smooth category where each manifold here is the total space of a smooth bundle over $S^1$ whose fiber is either of the following two.
				\begin{enumerate}
					\item An {\rm (}$m-1${\rm )}-dimensional homotopy sphere where $m \neq 5$.
					\item A $4$-dimensional standard sphere.
				\end{enumerate}
			\end{enumerate}
			\item
			\label{thm:1.2}
			Let $m$ be an arbitrary integer satisfying $m \geq 4$.
			An $m$-dimensional closed and simply-connected manifold $M$ admitting a special generic map $f:M \rightarrow {\mathbb{R}}^3$ is either of the following manifolds.
			\begin{enumerate}
			\item A homotopy sphere whose dimension is not $4$.
			\item A $4$-dimensional standard sphere.
				\item A manifold
				represented as a connected sum of smooth manifolds taken in the smooth category where each manifold here is the total space of a smooth bundle over $S^2$ whose fiber is either of the following two.
				\begin{enumerate}
				\item An {\rm (}$m-2${\rm )}-dimensional homotopy sphere where $m \neq 6$.
	      		\item A $4$-dimensional standard sphere.
				\end{enumerate}
			\end{enumerate}
			In the case $m=4,5$, the converse also holds where a fiber of each bundle is an {\rm (}$m-2${\rm)}-dimensional standard sphere and the bundles are linear.
			\item
			\label{thm:1.3}
		Both in the cases {\rm (}\ref{thm:1.1}{\rm )} and {\rm (}\ref{thm:1.2}{\rm )},
		if the manifold $M$ is not a homotopy sphere, then we have a special generic map as in Example \ref{ex:1} where the internal smooth bundle or the boundary linear bundle of it may not be trivial. In these cases, take $n_j=1$ and $n_j=2$ in Example \ref{ex:1}, respectively.
		\end{enumerate}
	\end{Thm}
	\begin{Def}[\cite{kitazawa0.1,kitazawa0.2,kitazawa0.3}]
		\label{def:3}
	Let $m \geq n \geq 1$ be integers and $f:M \rightarrow {\mathbb{R}}^n$ a fold map on an $m$-dimensional closed and connected manifold $M$.
	$f$ is a {\it round} fold map if either of the following two holds.
	\begin{enumerate}
		\item $n=1$, at distinct singular points of $f$ the values are distinct, and there exist a point in $\mathbb{R}$ such that $f^{-1}(a)$ has no singular points and a pair $(\Phi:f^{-1}((-\infty,a]) \rightarrow f^{-1}([a,+\infty)) ,{\phi}_{\mathbb{R}}:(-\infty,a] \rightarrow [a,+\infty))$ of diffeomorphisms such that the relation $f {\mid}_{f^{-1}([a,+\infty))} \circ \Phi={\phi}_{\mathbb{R}} \circ f {\mid}_{f^{-1}((-\infty,a])}$ (where the spaces of the targets of some of the maps are suitably restricted). $f^{-1}((-\infty,a])$ and $f^{-1}([a,+\infty))$ are called {\it axis manifolds}.
		
		\item $n \geq 2$ and the image $f(S(f))$ of the singular set is concentric. In other words, for some diffeomorphism ${\phi}_{{\mathbb{R}}^n}$ and an integer $l>0$, the relation
		$({\phi}_{{\mathbb{R}}^n} \circ f)(S(f))=\{x \in {\mathbb{R}}^n \mid 1 \leq ||x|| \leq l, ||x|| \in \mathbb{N}\}$ holds.
	\end{enumerate}
\end{Def}

Hereafter, for a {\it round} fold map $f:M \rightarrow {\mathbb{R}}^n$ with $n \geq 2$, we also assume the condition $f(S(f))=\{x \in {\mathbb{R}}^n \mid 1 \leq ||x|| \leq l, ||x|| \in \mathbb{N}\}$ by omitting the diffeomorphism ${\phi}_{{\mathbb{R}}^n}$. We can see this has no problem.

For such a round fold map, compose the restriction of $f$ to the preimage $f^{-1}(\{x \in {\mathbb{R}}^n \mid ||x|| \geq \frac{1}{2}\})$ with the canonical projection to $S^{n-1}$ mapping $x$ to $\frac{1}{||x||}x$, we have a smooth bundle. We call this a {\it global bundle} of $f$ and the fiber an {\it axis manifold} of $f$.
Let $C_j:=\{x \in f(S(f)) \mid ||x||=j\}$ for $j \in \mathbb{N}$ satisfying $1 \leq j \leq l$. Let $N_0(C_j):=\{x \in f(S(f)) \mid j-\frac{1}{2} \leq ||x|| \leq j+\frac{1}{2}\}$.
Instead, compose the restriction of $f$ to the preimage $f^{-1}(N_0(C_j))$ with the canonical projection to the unit sphere $S^{n-1}$ mapping $x$ to $\frac{1}{||x||}x$.
We call this a {\it local bundle at $C_j$} of $f$.
We can define the following.
	\begin{Def}
		\label{def:4}
	Let $m \geq n \geq 1$ be integers and $f:M \rightarrow {\mathbb{R}}^n$ a round fold map. 

	\begin{enumerate}
		\item If $n=1$ or $n \geq 2$ and a global bundle of $f$ is trivial, then $f$ is said to {\it have a globally trivial monodromy}.
		\item If $n=1$ or $n \geq 2$ and local bundle at each connected component of $f(S(f))$ of $f$ is trivial, then $f$ is said to {\it have componentwisely trivial monodromies}.
	\end{enumerate}
\end{Def}
We can also check that canonical projections of unit spheres are also round fold maps having globally trivial monodromies and componentwisely trivial monodromies as another exercise.
	
	\begin{Thm}[\cite{kitazawa0.1,kitazawa0.2,kitazawa0.5}]
		\label{thm:2}
		Let $m>n \geq 1$ be integers.
		Let $\Sigma$ be an {\rm (}$m-n${\rm )}-dimensional homotopy sphere which is not a $4$-dimensional exotic sphere. A closed and connected manifold $M$ admits a round fold map $f:M \rightarrow {\mathbb{R}}^n$ having a globally trivial monodromy and satisfying the following conditions if and only if it is the total space of a smooth bundle over $S^{n}$ whose fiber is diffeomorphic to $\Sigma$.
			
			 \begin{enumerate}
			 	\item $n=1$ and the singular set of $f$ consists of exactly four points.
			 	Let the three connected components of $f(M) \bigcap (\mathbb{R}-f(S(f)))$ be denoted by $R_1$, $R_2$, and $R_3$, respectively. We can define so that the closures of $R_{j_1}$ and $R_{j_2}$ intersect if and only if $(j_1,j_2) \neq (1,3),(3,1)$.
			 	 The preimage of point in $R_j$ for $f$ is diffeomorphic to $S^{m-n}$ for $j=1,3$ and $\Sigma \sqcup \Sigma$ for $j=2$. 
			\item $n \geq 2$ and $l=2$ in Definition \ref{def:3}. The preimage of a point of $x \in {\mathbb{R}}^n$ for $f$ satisfying $||x|| < 1$ is diffeomorphic to $\Sigma \sqcup \Sigma$. The preimage of a point of $x \in {\mathbb{R}}^n$ satisfying $1<||x||<2$ for $f$ is diffeomorphic to an {\rm (}$m-n${\rm )}-dimensional standard sphere.
			\end{enumerate}	
\end{Thm}

Hereafter, we present results on round fold maps which are not crucial in our paper essentially.

	\begin{Thm}[\cite{kitazawa0.1,kitazawa0.2,kitazawa0.5}]
		\label{thm:3}
	Let $m>n \geq 1$ be integers.
Let $l$ be an arbitrary integer greater than $1$. A closed and connected manifold $M$ represented as a connected sum of $l-1$ manifolds represented as the total spaces of smooth bundles over $S^n$ whose fibers are {\rm (}$m-n${\rm )}-dimensional standard spheres admits a round fold map $f:M \rightarrow {\mathbb{R}}^n$having a componentwisely trivial monodromies and satisfying either of the following conditions.
	\begin{enumerate}
		\item $n=1$ in Definition \ref{def:3} and the singular set of $f$ consists of exactly $2l$ points. Let the family of exactly $2l-1$ connected components of $f(M) \bigcap (\mathbb{R}-f(S(f)))$ be denoted by $\{R_j\}_{j=1}^{2l-1}$. We can define so that the closures of $R_{j_1}$ and $R_{j_2}$ intersect if and only if $|j_1-j_2|=1$. Furthermore, the preimage of a point in $R_j$ is the disjoint union of exactly $m_j$ {\rm (}$m-n${\rm )}-dimensional standard spheres where $m_j:=\min\{2l-1-j,j\}$ for $j \neq l$ and $m_j:=j$ for $j=l$.
		\item $n \geq 2$ in Definition \ref{def:3} and the singular set of $f$ consists of exactly $l$ connected components. Let the family of all $l$ connected components of $f(M) \bigcap ({\mathbb{R}}^n-f(S(f)))$ be denoted by $\{R_j\}_{j=1}^{l}$ where $R_j:=\{x \in {\mathbb{R}}^n \mid j-1 \leq ||x|| \leq j\}$. The preimage of a point in $R_j$ is the disjoint union of exactly $l+1-j$ {\rm (}$m-n${\rm )}-dimensional standard spheres.
	\end{enumerate}
	If $m \geq 2n$ holds, then the converse also holds.
\end{Thm}
	\begin{Thm}[\cite{kitazawa0.4}]
		\label{thm:4}
	Let a closed and connected manifold $M_0$ admit a round fold map $f_0:M_0 \rightarrow {\mathbb{R}}^{n}$ having componentwisely trivial monodromies. 
	Let $M$ be the total space of a smooth bundle over $M_0$ whose fiber is connected. For each connected component $C$ of $f(S(f))$, consider a small closed tubular neighborhood $N(C)$ of $C$. Assume that the restriction of the bundle $M$ to each $f^{-1}(N(C))$ is trivial.  
	Then $M$ admits a round fold map $f:M \rightarrow {\mathbb{R}}^{n}$ having componentwisely trivial monodromies.
\end{Thm}

\begin{Ex}[\cite{kitazawa0.4}]
	\label{ex:2}
Consider suitable linear bundles over manifolds admitting round fold maps in Theorems \ref{thm:2} and \ref{thm:3}. We can apply Theorem \ref{thm:4} to obtain various maps and manifolds. In general we have infinitely many maps and manifolds.
\end{Ex}

	\section{Decompositions of smooth or PL manifolds in our paper.}
	As decompositions of manifolds, {\it Heegaard splittings} of $3$-dimensional closed and connected manifolds are pioneering notions.
	They are divided into two copies of a $3$-dimensional {\it 1-handlebody}.
	
	Let $m \geq 2$ be an arbitrary integer.
	An $m$-dimensional {\it 1-handlebody} (in the orientable case) is a smooth manifold diffeomorphic to a unit disk or a manifold represented as a connected sum of manifolds each of which is the product of a copy of the circle $S^1$ and the unit disk $D^{m-1}$: we consider the connected sum in the smooth category. Of course this is regarded as a PL manifold in the canonical way uniquely. In considering such manifolds in the PL or the piecewise category, we regard 1-handlebody as a manifold PL homeomorphic to such a manifold whereas in the smooth category we can give differentiable structures different from the standard differentiable structures as above.
	
	In \cite{kitazawa2}, we have defined several classes of decompositions of smooth or PL manifolds respecting the existing classes such as ones in \cite{rubinsteintillmann1,rubinsteintillmann2}, generalizing the dimensions of the manifolds in the theory of Heegaard splittings of $3$-dimensional closed and connected manifolds and $4$-dimensional cases, discussed in a pioneering study \cite{gaykirby}. We also presented non-trivial examples for these introduced classes there.
	
	We concentrate on cases which are general and not wild. In the previous paper, we may ignore our global arguments in our present paper (in most situations) and we admit wildness. There we do not give related explicit wild examples which seem to be meaningful in some senses. Discoveries and presentations of such examples seem to be also difficult in general. 
	
	To define a {\it smooth} or {\it PL} {\it globally neat} ({\it GN}) {\it decompositions} and {\it globally almost neat multibranched} {\it decompositions}.
	See also an argument on trisections (or slightly general decompositions) of standard spheres of \cite{lambertcolemiller}. 

We introduce a natural decomposition of a copy of the unit disk $D^m$ into finitely many copies of the unit disk $D^m$.
		
		 Let $m$ be a positive integer.
			
		 Let $k_1$ be an arbitrary non-negative integer satisfying $0 \leq k_1 \leq m-1$ and let us define ${{\mathbb{R}}^m}_{k_1,\leq}$ as the set of all elements of $(x_1,\cdots,x_m) \in {\mathbb{R}}^m$ such that for any integer $1 \leq j \leq k_1$, $x_j \geq 0$ and that $x_{k_1+1} \leq 0$.
		 
		 Let $k_2$ be an arbitrary non-negative integer satisfying $0 \leq k_2 \leq m$ and let us define ${{\mathbb{R}}^m}_{k_2}$ as the set of all elements $(x_1,\cdots,x_m) \in {\mathbb{R}}^m$ such that for $1 \leq j \leq k_2$, $x_j \geq 0$. 
		 
		 Let $k_3$ be an arbitrary non-negative integer satisfying $0 \leq k_3 \leq m-2$. Let $i>0$ be an arbitrary integer and $i^{\prime}$ an arbitrary integer satisfying $0 \leq i^{\prime}\leq i$. Let us define ${{\mathbb{R}}^m}_{k_3,i,i^{\prime}}$ as the set of all elements $(x_1,\cdots,x_m) \in {\mathbb{R}}^m$ such that for $1 \leq j \leq k_3$, $x_{k_3} \geq 0$ and that $(x_{k_3+1},x_{k_3+2}) \in (r \cos t,r\sin t)$ for some $r \geq 0$ and $i^{\prime}\frac{\pi}{i+1} \leq t \leq (i^{\prime}+1)\frac{\pi}{i+1}$.
		 
	\begin{Def}
		\label{def:5}
			We abuse the integers just before.
		\begin{enumerate}
			
			\item \label{def:5.1} The family $\{D^m \bigcap {{\mathbb{R}}^m}_{j^{\prime}-1,\leq}\}_{j^{\prime}=1}^{k_1+1} \sqcup \{D^m \bigcap {{\mathbb{R}}^m}_{k_2}\}$ is said to define a {\it standard simple multisection of the unit disk $D^m$} of {\it degree $k_1+2 \leq m+1$} where $k_2=k_1+1$.
			\item \label{def:5.2}
			The family $\{D^m \bigcap {{\mathbb{R}}^m}_{j^{\prime}-1,\leq}\}_{j^{\prime}=1}^{k_3+1} \sqcup \{D^m \bigcap {{\mathbb{R}}^m}_{i,i^{\prime}-1}\}_{i^{\prime}=1}^{i+1}$ is said to define a {\it standard almost simple multibranched multisection of the unit disk $D^m$} of {\it degree $k_3+i+2$}.
			\item \label{def:5.3}
		    The family $\{\partial D^m \bigcap {{\mathbb{R}}^m}_{j^{\prime}-1,\leq}\}_{j^{\prime}=1}^{k_1+1} \sqcup \{\partial D^m \bigcap {{\mathbb{R}}^m}_{k_2}\}$ is said to define a {\it standard simple multisection of the unit sphere $S^{m-1}$} of {\it degree $k_1+2 \leq m+1$} where $k_2=k_1+1$.
		    \item \label{def:5.4}
		      The family $\{\partial D^m \bigcap {{\mathbb{R}}^m}_{j^{\prime}-1,\leq}\}_{j^{\prime}=1}^{k_3+1} \sqcup \{\partial D^m \bigcap {{\mathbb{R}}^m}_{i,i^{\prime}-1}\}_{i^{\prime}=1}^{i+1}$ is said to define a {\it standard almost simple multibranched multisection of the unit sphere $S^{m-1}$} of {\it degree $k_3+i+2$}.
			\end{enumerate}
	\end{Def}
We can define the following notions as Definition \ref{def:6} so that the latter notion is a generalization of the former notion. 
\begin{Def}
	\label{def:6}
	\begin{enumerate}
		\item Let $X$ be a smooth (PL) closed and connected manifold. Let $k>0$ be a positive integer. Suppose that a family $\{X_j\}_{j=1}^k$ of smooth (resp. PL) compact and connected submanifolds whose dimensions are $\dim X$ enjoying the following properties exists.
		\begin{enumerate}
			\item ${\rm Int}\ X_{j_1} \bigcap {\rm Int}\ X_{j_2}$ is empty for distinct two numbers $j_1$ and $j_2$.
			\item The intersection ${\bigcap}_{j=1}^k X_j$ is not empty.
		\end{enumerate} 
	Then $\{X_j\}_{j=1}^k$ is said to define an {\it admissible decomposition} of $X$ of degree $k$.
	\item Let $X$ be a smooth (PL) compact and connected manifold. Let $k>0$ be a positive integer. Suppose that a family $\{X_j\}_{j=1}^k$ of smooth (resp. PL) compact and connected submanifolds whose dimensions are $\dim X$ enjoying the following properties exists.
	\begin{enumerate}
		\item ${\rm Int}\ X_{j_1} \bigcap {\rm Int} X_{j_2}$ is empty for distinct two numbers $j_1$ and $j_2$.
		\item The intersection ${\bigcap}_{j=1}^k X_j$ is not empty.
		\item The family $\{{\partial X}_i \bigcap X_j\}_{j=1}^k$ defines an admissible decomposition of $\partial X$ for any connected component ${\partial X}_i$ of $\partial X$.
	\end{enumerate}
	Then $\{X_j\}_{j=1}^k$ is said to define an {\it admissible decomposition} of $X$ of degree $k$.
	\end{enumerate}

\end{Def}
Decompositions od Definition \ref{def:5} are admissible smooth and PL decompositions (where in the cases (\ref{def:5.3}) and \ref{def:5.4}, $m>1$ is assumed).
\begin{Def}\label{def:7}
	An (admissible) smooth or PL decomposition of $X$ defined by $\{X_j\}_{j=1}^k$
is said to be a {\it multisection} if each submanifold is a $1$-handlebody where $1$-handlebodies are considered in the category we consider as presented in the definition of a $1$-handlebody.
	\end{Def}
	
We can define an equivalence relation on the set of all admissible smooth (PL) decompositions as follows.

\begin{Def}\label{def:8}
Two admissible decompositions are {\it smooth} ({\it PL}) {\it equivalent} if there exists a diffeomorphism (resp. piecewise smooth homeomorphism) mapping each submanifold in the family into another submanifold in the family (resp. where corners may be created or eliminated suitably).
\end{Def}

Remember also that smooth manifolds are regarded as PL manifolds in a unique way.

\begin{Def}\label{def:9}
	For a smooth (PL) admissible decomposition of a compact and connected manifold $X$ defined by $\{X_j\}_{j=1}^k$,
\begin{enumerate}
	\item \label{def:9.1}
	 For each point of $p$ and a suitable small smoothly embedded copy $D_p$ of the unit disk $D^{\dim X}$ containing $p$ in the interior, suppose that $\{D_p \subset X_j\}_{j=1}^k$ defines an admissible decomposition of $D_p$ smooth (resp. PL) equivalent to the stand simple multisection of the unit disk $D^{\dim X}$. 
	Then $\{X_j\}_{j=1}^k$ is said to define a {\it globally neat} or {\it GN} decomposition of $X$.
		Here copies of the unit disk may be embedded as manifolds having corners whereas they may be eliminated depending on the situations. 
	\item \label{def:9.2} For each point of $p$ and a suitable small smoothly embedded copy $D_p$ of the unit disk $D^{\dim X}$ containing $p$ in the interior, suppose that $\{D_p \subset X_j\}_{j=1}^k$ defines an admissible decomposition of $D_p$ smooth (resp. PL) equivalent to the standard almost simple multibranched multisection of the unit disk $D^{\dim X}$. 
	Then $\{X_j\}_{j=1}^k$ is said to define a {\it globally almost neat multibranched} or {\it GANM} decomposition of $X$.
	Here copies of the unit disk may be embedded as manifolds having corners whereas they may be eliminated depending on the situations. 
\end{enumerate}	
\end{Def}
These notions are regarded as revised notions in some notions of \cite{kitazawa8}. They also respect $3$-dimensional cases mainly. For $3$-dimensional related studies, \cite{koenig} is a pioneering one, followed by \cite{itoogawa} for example. \cite{ogawa, sakatamishinaogawaishiharakodaozawashimokawa} are also respected for example as related $3$-dimensional studies.
The following also presents our new notion.

We need the notion of a {\it PL bundle}, defined as a bundle whose fiber is a polyhedron and whose structure group consists of piecewise smooth homeomorphisms.


\begin{Def}\label{def:10}
	An admissible smooth (PL) decomposition of a compact and connected manifold $X$ of dimension $\dim X>1$ defined by $\{X_j\}_{j=1}^k$ is said to be doubled {\it GNANM} or {\it DGNANM} if at least one of the following hold. This notion is inductively defined.
	\begin{enumerate}
		\item \label{def:10.1}
		 It is a GN or GANM decomposition of a closed and connected manifold $X$.
		\item \label{def:10.2}
		${\bigcup}_{j=1}^k \partial X_j$ is a polyhedron PL homeomorphic to one obtained by gluing the following three polyhedra and $X$ is a smooth (resp. PL) manifold diffeomorphic to one presented later here.
		\begin{enumerate}
			\item The product of ${\bigcup}_{j=1}^{k^{\prime}} \partial Y_{j}$ and a smooth (resp. PL) compact and connected manifold $F_{Y_1}$ whose boundary is non-empty and diffeomorphic (resp. PL homeomorphic) to a closed manifold $F$ where $\{Y_{j}\}_{j=1}^{k^{\prime}}$ defines an admissible smooth (resp. PL) DGNANM decomposition of a closed and connected manifold $Y$.
				\item The product of ${\bigcup}_{j=1}^{k^{\prime}} \partial Y_{j}$ and a smooth (resp. PL) compact and connected manifold $F_{Y_2}$ whose boundary is non-empty and diffeomorphic (resp. PL homeomorphic) to $F$.
			\item A smooth (PL) closed manifold $P$ of dimension $\dim X-1$.
		\end{enumerate}
		Furthermore, the polyhedron is
		 PL homeomorphic to a polyhedron obtained by gluing the polyhedra PL homeomorphic to ${\bigcup}_{j=1}^{k^{\prime}} (\partial Y_{j}) \times F_{Y_1}$ and ${\bigcup}_{j=1}^{k^{\prime}} (\partial Y_{j}) \times F_{Y_2}$ to a subpolyhedron of $P$, PL homeomorphic to ${\bigcup}_{j=1}^{k^{\prime}} (\partial Y_{j}) \times F$, by piecewise a smooth homeomorphism from the subpolyhedron ${\bigcup}_{j=1}^{k^{\prime}} (\partial Y_{j}) \times F \subset {\bigcup}_{j=1}^{k^{\prime}} (\partial Y_{j}) \times F_{Y_{j}}$ for $i=1,2$ where $F$ is regarded as the boundary of $F_{Y_i}$ in the canonical way. Furthermore, $X$ is diffeomorphic (resp. PL homeomorphic) to a manifold obtained by attaching $Y \times F_{Y_1}$ and $Y \times F_{Y_2}$ along the boundaries $Y \times F \subset Y \times F_{Y_1}$ and $Y \times F \subset Y \times F_{Y_2}$ by a diffeomorphism (resp. piecewise smooth homeomorphism) regarded as the product of a diffeomorphism (resp. piecewise smooth homeomorphism) from $\partial F_{Y_1}$ onto $\partial F_{Y_2}$ and the identity map on $Y$ in such a way that gives the polyhedron before by gluing the two subpolyhedra embedded canonically in the two manifolds here. 
		\item \label{def:10.3}
		$\dim X>2$ is assumed in this case. There exist a smooth (resp. PL) decomposition of a closed and connected manifold $X_0$ of dimension $\dim X=\dim X_0>2$ containing $X_j$ as a smooth (resp. PL) compact and connected submanifold and an admissible smooth (resp. PL) decomposition of $X_0$ defined by a family $\{X_{0,j}\}_{j=1}^k$ of submanifolds and enjoying the relation $X_j=X_{0,j} \bigcap X$. Furthermore, $\{\partial X \bigcap X_{0,j}\}_{j=1}^k$ defines a smooth (resp. PL) DGNANM decomposition of $\partial X$. In the smooth case, ${\bigcap}_{j \in J} X_{0,j}$ is always a smooth compact submanifold of $X_0$ for any non-empty subset of the set of all integers from $1$ to $k$ and the intersection of the tangent vector space of ${\bigcap}_{j \in J} X_{0,j}$ at $p \in {\bigcap}_{j \in J} X_{0,j} \bigcap \partial X$ and that of $\partial X$ at $p$ is the trivial vector space. In other words, the transversality is satisfied: transversality is conditions or properties on differentiable maps and see \cite{golubitskyguillemin} for example.
		
		\item \label{def:10.4}
		 ${\bigcup}_{j=1}^k \partial X_j$ is a polyhedron PL homeomorphic to one obtained by taking the following five steps and $X$ is a smooth (resp. PL) manifold as presented here.\\
		 \ \\
		 STEP 10-4-1 Prepare the following polyhedra.
		\begin{enumerate}
			\item \label{def:10.4.1} 
			The product of ${\bigcup}_{j=1}^{k^{\prime}} \partial Y_{j}$ and a compact and connected PL manifold $F_{Y_1}$ whose boundary is non-empty and diffeomorphic to a closed manifold $F$ where $\{Y_{j}\}_{j=1}^{k^{\prime}}$ defines an admissible smooth (resp. PL) DGNANM decomposition of a compact and connected manifold $Y$ whose boundary $\partial Y$ is not empty.
			\item \label{def:10.4.2}
			 The product of ${\bigcup}_{j=1}^{k^{\prime}} \partial Y_{j}$ and a PL compact and connected manifold $F_{Y_2}$ whose boundary is non-empty and diffeomorphic to $F$.
			\item \label{def:10.4.3}
			A PL compact manifold $P$ of dimension $\dim X-1$ whose boundary is PL homeomorphic to $Y \times F$.
			\end{enumerate}
			\ \\
			STEP 10-4-2 Construct a polyhedron from polehedra in STEP 10-4-1. \\
			Construct a polyhedron $P_{F,Y,{\rm a,b,c}}$ by gluing the polyhedra PL homeomorphic to ${\bigcup}_{j=1}^{k^{\prime}} (\partial Y_{j}) \times F_{Y_1}$ and ${\bigcup}_{j=1}^{k^{\prime}} (\partial Y_{j}) \times F_{Y_2}$ to a subpolyhedron of $P$ in ${\rm Int}\ P$, PL homeomorphic to ${\bigcup}_{j=1}^{k^{\prime}} (\partial Y_{j}) \times F$, by a piecewise smooth homeomorphism from the subpolyhedron ${\bigcup}_{j=1}^{k^{\prime}} (\partial Y_{j}) \times F \subset {\bigcup}_{j=1}^{k^{\prime}} (\partial Y_{j}) \times F_{Y_{j}}$ for $i=1,2$ where $F$ is regarded as the boundary of $F_{Y_i}$ in the canonical way. In this step, we use our polyhedra in (\ref{def:10.4.1}), (\ref{def:10.4.2}) and (\ref{def:10.4.3}). Let $P_{F,Y,{\rm a,b,c}}$ denote this resulting polyhedron. \\
				\ \\
			STEP 10-4-3 Embed $P_{F,Y,{\rm a,b,c}}$ into the product of a closed manifold $E_Y$ and $Y$. \\
			We define a manifold $E_Y$ as follows.
			\begin{enumerate}
				\item \label{def:10.4.3.1}
				$E_Y$ is a PL compact manifold and there exists an admissible PL decomposition of degree $2$ defined by the pair $\{E_{Y_j} \subset E_Y\}_{j=1}^2$.
				\item \label{def:10.4.3.2}
				The pair $\{F_{Y_j}\}_{j=1}^2$ is regarded to  define an admissible PL decomposition of degree $2$ of the boundary $\partial E_Y$ of $E_Y$. Furthermore, $F_{Y_i}$ is identified with $\partial E_{Y} \bigcap E_{Y_i}$ for $i=1,2$.
			\end{enumerate}
		We can realize $P_{F,Y,{\rm a,b,c}} \subset \partial E_Y \times Y$ as a subpolyhedron in a natural way and we do so.\\
		\ \\
		STEP 10-4-4 We prepare another manifold $E_Y \times Y$ and a sunpolyhedron $P_{E,Y}$. \\
		We define a subpolyhedron $P_{E,Y}$ as the union of $E_Y \times (({\bigcup}_{j=1}^{k^{\prime}} \partial Y_{j}) \bigcap \partial Y) \subset E_Y \times \partial Y$ and $(\partial E_{Y_i}-{\rm Int}\ F_{Y_i}) \times \partial Y \subset E_Y \times \partial Y$.\\
		\ \\
		STEP 10-4-5 We glue the polyhedra and their outer manifolds via a bundle isomorphism between the outer manifolds, regarded as trivial bundles and obtain our desired polyhedron. \\
		We glue the trivial smooth (resp. PL) bundles $\partial E_Y \times Y$ over $\partial E_Y$ and $E_Y \times \partial Y$. This gives a desired manifold diffeomorphic (resp. PL homeomorphic) to $X$. We glue them via a bundle isomorphism between the trivial smooth (resp. PL) bundles over $\partial E_Y$ whose fibers are diffeomorphic to (resp. PL homeomorphic to) $\partial Y$ induced on the boundaries.
		We can glue the subpolyhedra $P_{F,Y,{\rm a,b,c}}$ and $P_{E,Y}$ together in this operation and presents our desired polyhedron ${\bigcup}_{j=1}^k \partial X_j$.

\item \label{def:10.5} $X$ and $\{X_j\}_{j=1}^k$
are obtained in the following way from two suitable smooth (resp. PL)
DGNANM decompositions in the following. One of these two is of a compact and connected manifold $X_1$ and defined by a family of submanifolds $\{X_{1,j}\}_{j=1}^k$ and the other is of a compact and connected manifold $X_2$ and defined by a family $\{X_{2,j}\}_{j=1}^k$ of submanifolds. Furthermore, $\dim X_i=\dim X$ holds and the following properties are enjoyed.
\begin{itemize}
	\item 
For each point $p_i$ of ${\rm Int}\ X_i$, let $\{X_{i,j}\}_{j \in J_{i,p}}$ be the set of all submanifolds in the family $\{X_{i,j}\}_{j=1}^k$ containing $p_i$ where $p_i$ may not be in the interior of each submanifold. Then we can take a suitable small copy $D_{J_{i,p}} \subset {\rm Int}\ X_i$ of the unit disk $D^{\dim X}$ containing $p_i$ in the interior in such a way that $\{D_{J_{i,p}} \bigcap X_{i,j}\}_{j \in J_{i,p}}$ and $\{\partial D_{J_{i,p}} \bigcap X_{i,j}\}_{j \in J_{i,p}}$ define DGNANM disk decompositions. 
\item The two disk decompositions $\{D_{J_{1,p}} \bigcap X_{1,j}\}_{j \in J_{1,p}}$ and $\{D_{J_{2,p}} \bigcap X_{2,j}\}_{j \in J_{2,p}}$
are smooth (resp. PL) equivalent. A diffeomorphism (resp. piecewise smooth homeomorphism) ${\Phi}_{X,1,2}:D_{J_{1,p}} \rightarrow D_{J_{2,p}}$ maps each submanifold into a submanifold in the other family and this makes them smooth (resp. PL) equivalent.
\item As in (\ref{def:10.3}) here, we can define a smooth (resp. PL) DGNANM decomposition of $X_i-{\rm Int}\ D_{J_i,p}$, defined by $\{(X_i-{\rm Int}\ D_{J_{i,p}}) \bigcap X_{i,j}\}_{j \in J_{i,p}}$ for $i=1,2$.
\end{itemize}
${\Phi}_{X,1,2}$ gives a connected sum of $X_1$ and $X_2$ by the definition in the smooth category (resp. PL category) and this presents our desired smooth (resp. PL) manifold $X$. 
\end{enumerate}
\end{Def}
\begin{Def}
	\label{def:11}
		An admissible smooth (PL) decomposition of a compact and connected manifold $X$ of dimension $\dim X>1$ defined by $\{X_j\}_{j=1}^k$ is said to be {\it DGNANM outside a smooth }{\rm (}{\it resp. PL}{\rm )} {\it submanifold $F$} with no boundary if we can choose a small closed tubular (resp. regular) neighborhood $N(F)$ of a smooth (resp. PL) closed submanifold $F$ with no boundary in ${\rm Int}\ X$ in such a way that $\{X_j \bigcap (X-{\rm Int}\ N(F))\}_{j=1}^k$ defines a DGNANM decomposition of $X-{\rm Int}\ N(F)$. 
\end{Def}
In our Main Theorems, new notions in Definitions \ref{def:10} and \ref{def:11} are important.
Hereafter, an admissible decomposition or multisection is also called a {\it disk decomposition} if the family of the submanifolds consists of copies of a unit disk. Admissible decompositions in Definition \ref{def:5} are all disk decompositions.

By the definition and considering the local structures, we can easily know the following fundamental property. We only present important ingredients in a proof.
\begin{Prop}
		\label{prop:3}
	GN or GANM decompositions are always DGNANM decompositions in Definition \ref{def:10} {\rm (}\ref{def:10.3}{\rm )}.
	\end{Prop}
\begin{proof}[Important ingredients in a proof]
Due to Definitions \ref{def:5}	and \ref{def:9} and the local structures especially, we can consider a double of the given compact and connected manifold in the case where the boundary is non-empty and we have a GN or GNAM decomposition on a closed and connected manifold.
\end{proof}
\begin{Prop}
	\label{prop:4}
	For each point $p$ of $X$ in Definition \ref{def:10} {\rm (}Definition \ref{def:11}{\rm )}, let $\{X_{j}\}_{j \in J_p}$ be the set of all submanifolds in the family $\{X_{j}\}_{j=1}^k$ containing $p$ {\rm (}$p \in X-{\rm Int}\ N(F)${\rm )}. Then we can take a suitable small copy $D_{J_p}$ of the unit disk $D^{\dim X}$ containing $p$ in the interior in such a way that $\{D_{J_p} \bigcap X_{j}\}_{j \in J_p}$ and $\{\partial D_{J_p} \bigcap X_{j}\}_{j \in J_p}$ define DGNANM disk decompositions. Moreover, for $p \in {\rm Int}\ X$, we can define a DGNANM decomposition of $X-{\rm Int}\ D_{J_p}$ as we do in Defintion \ref{def:10} {\rm (}\ref{def:10.3}{\rm )} for a suitably chosen $D_{J_p}$. 
\end{Prop}
We only present important ingredients in a proof of Proposition \ref{prop:4}.
\begin{proof}[Some important ingredients in a proof of Proposition \ref{prop:4}]
Fundamental technique is considering local structures and applying inductions, for example. More precisely, the fact that some our important($\dim X-1$)-dimensional compact and connected manifolds such as $P$ divide each of our (nice) disk decompostions into two copies is one of key ingredient.
\end{proof}

We present several known studies on GN cases. Note again that GN decompositions, GNANM ones and DGNANM ones are introduced by the author first. In \cite{kitazawa8}, the author presents a pioneering study on good definitions of such classes of decompositions. Some of the classes respect existing studies on nice multisections or more general decompositions. 
\begin{Thm}
	\label{thm:5}
	\begin{enumerate}
		\item 
		\label{thm:5.1}
		$3$-dimensional closed, connected and orientable manifolds always admit smooth and PL GN multisections of degree $2$. They are Heegaard splittings.
		\item \label{thm:5.2}
		$3$-dimensional closed and connected manifolds always admit smooth and PL GN disk decomposition of degree $4$. They are shown from the theory of Heegaard splittings. For the $3$-dimensional closed, connected and non-orientable manifolds, Heegaard splittings are obtained. However, {\rm $1$-handlebodies in the non-orientable} case are defined as non-orientable variants. Hereafter and in our paper, we do not consider non-orientable cases unless otherwise stated.
		\item \label{thm:5.3}
		$4$-dimensional and $5$-dimensional smooth, closed, connected and orientable manifolds always admit smooth and PL GN multisections of degree $3$. They also admit smooth and PL GN disk multisections of degree $6$.
		See \cite{gaykirby} for the $4$-dimensional case and the $5$-dimensional case is due to \cite{lambertcolemiller}
	
		\item \label{thm:5.4}
		{\rm (}$2k-1${\rm )}-dimensional {\rm PL}, closed, connected and orientable manifolds always admit PL GN multisections of degree $k$ for any integer $k \geq 1$. They also admit PL GN disk decompositions of degree $2k$. {\rm (}$2k${\rm )}-dimensional closed, connected and orientable manifolds always admit PL GN multisections of degree $k+1$ for any integer $k \geq 1$. They also admit PL GN disk decompositions of degree $2(k+1)$.
		This is due to \cite{rubinsteintillmann1,rubinsteintillmann2}.
	\end{enumerate} 
\end{Thm}
	We do not explain about GANM cases explicitly. We note some. 
	
	\cite{ogawa} is on multisections of $3$-dimensional closed, connected and orientable manifolds which are not GN. They are decomposed by so-called {\it multibranched surfaces}. \cite{asano,kobayashi} consider such cases in $4$-dimensional cases via generic smooth maps into the plane. 
	
	For example, we can see that in Definition \ref{def:10} (\ref{def:10.2}) and (\ref{def:10.4}), consider the case the admissible decomposition of $Y$ is of degree $2$. We immediately have a GANM decomposition of $X$ which is not GN in general.
\section{Lusternik-Schnirelmann category numbers.}	
	We define the so-called {\it Lusternik-Schnirelmann category number} of a topological space, theory related to which is presented systematically in \cite{cornealuptonopreatanre} for example.
	
	The {\it Lusternik-Schnirelmann category number} of a topological space $X$ is defined as ${\rm cat}(X):= {\rm cat}^{\prime}(X)-1$ where ${\rm cat}^{\prime}(X)$ is as follows: ${\rm cat}^{\prime}(X)$ is the minimal number in the set of all numbers of subsets whose inclusions into $X$ are null-homotopic and which cover $X$ if the number exists. If the number does not exist, then the number is defined as $+\infty$. 
	 
	 For a manifold $X$, ${\rm cat}(X) \leq \dim X$ holds.
	  If we change subsets whose inclusions into $X$ are null-homotopic into ones which are homeomorphic to disks and add suitable conditions on configurations of the disks in $X$ for example, then such a number minus $1$ gives an upper bound of ${\rm cat}(X)$.
	  Numbers defined in such ways are in considerable cases smaller than or equal to the dimensions of the manifolds. 
	  See \cite{cavicchioli,luft} for example. 
	  
	  For a closed and connected manifold $X$, ${\rm cat}(X)=1$ if and only if $X$ is a sphere.
	  It is not so difficult to know ${\rm cat}(X)$ in the case where the dimension of $X$ is at most $2$. 
	  	  For explicit theory related to Morse functions and (smooth) manifolds, see also \cite{saeki3}. Related to Theorem \ref{thm:5}, for a $3$-dimensional closed and connected manifold $X$, ${\rm cat}(X) \leq 3$. Furthermore, \cite{gomezlarranagafgonzalezacuna} completely determines ${\rm cat}(X)$ for any $3$-dimensional closed, connected and orientable manifold $X$.
	  	  
	  	  Such numbers are in considerable cases studied by methods of homotopy theory and important mainly in related regions. In our paper, we present related observation by explicit smooth maps and give an explicit geometric and constructive study. 
\section{On Main Theorems.}	
\begin{Thm}[Main Theorem \ref{thm:1} again]
	\label{thm:6}
		Let $M_0$ be a smooth closed and connected manifold of dimension $m_0>0$ admitting an {\rm admissible smooth (PL) decomposition} of {\rm degree $k$} {\rm defined by a family $\{M_{0,j}\}_{j=1}^{k_0}$ of submanifolds}. 
	Let $M$ be the total space of a linear bundle over $M_0$ whose fiber is the $k$-dimensional unit sphere and which is reduced by degree $k^{\prime}$ with integers $k>0$ and $k^{\prime}>0$ satisfying $k-k^{\prime}>0$.
	
	Then, $M$ admits an admissible smooth {\rm (}resp. PL{\rm )} decomposition of degree $2k_0$.
	 Furthermore, this is as follows.
	 
	 \begin{enumerate}
	 	\item $M$ admits a special generic map $f:M \rightarrow N$ into some {\rm (}$m_0+k^{\prime}${\rm )}-dimensional connected and non-closed manifold $N$ whose image is a smoothly embedded manifold diffeomorphic to $M_0 \times D^{k^{\prime}}$.
	 	Here we decompose $D^{k^{\prime}}$ into two copies ${D^{k^{\prime}}}_1$ and ${D^{k^{\prime}}}_2$ of the unit disk $D^{k^{\prime}}$ glued via a diffeomorphism from a smoothly embedded copy of $D^{k^{\prime}-1}$ in the boundary $\partial {D^{k^{\prime}}}_1 \subset {D^{k^{\prime}}}_1$ onto a smoothly embedded copy of $D^{k^{\prime}-1}$ in the boundary $\partial {D^{k^{\prime}}}_2 \subset {D^{k^{\prime}}}_2$ to reconstruct $D^{k^{\prime}}${\rm :} in other words we consider a disk decomposition being a case of Proposition \ref{prop:3} for example. 
	 	\item The admissible decomposition is defined by the family $\{f^{-1}(M_{0,j} \times {D^{k^{\prime}}}_{1})\}_{j=1}^{k_0} \sqcup \{f^{-1}(M_{0,j} \times {D^{k^{\prime}}}_{2})\}_{j=1}^{k_0}$.
	 \end{enumerate}
	
	In the case where the given decomposition is a {\it GN} decomposition, the resulting decomposition is {\it GANM}. In the case where the given decomposition is {\it DGNANM}, so is the resulting one. In the case where the given decomposition is a multisection with the structure group of the given linear bundle being reduced to a so-called {\rm rotation group}, the resulting decomposition is also a multisection. In the case where the given decomposition is a {\it disk decomposition}, so is the resulting decomposition.
	\end{Thm}
\begin{proof}
	We construct a desired special generic map by applying Proposition \ref{prop:2} where $\bar{f}_N$ is a smooth embedding from $M_0 \times D^{k^{\prime}}$ into $N:=M_0 \times {\mathbb{R}}^{k^{\prime}}$.
	
	For the canonical projection of the unit sphere $S^k$ onto ${\mathbb{R}}^{k^{\prime}}$ as presented, we have a natural action of the orthogonal group of degree $k-k^{\prime}+1$ acting on the unit sphere preserving the values of the canonical projection. This canonical projection is, as presented, special generic and the image is the unit disk. We can see we have a desired special generic map. We can also see that we have a desired admissible smooth or PL decomposition.
	
	By considering Definition \ref{def:10} (\ref{def:10.2}) and the structures of the maps and the admissible decompositions for example, we have results on GN, GANM, and DGNANM decompositions.
	
	Note also that the decomposition given by the family $\{M_{0,j} \times {D^{k^{\prime}}}_1\}_{j=1}^{k_0} \sqcup \{M_{0,j} \times {D^{k^{\prime}}}_2\}_{j=1}^{k_0}$ is GANM (DGNANM) if the given admissible decomposition is GN and of degree $2$ (resp. DGNAMN). In the case of a GN decompotion of degree $2$, we have this from the definitions easily. For this, it is important to remember a short argument in the end of the third section. In the case of a DGNAMN decomposition, consider a double of the given manifold for example.
	
	We explain about results on disk decompositions and multisections. $M_{0,j}$ is a copy of the unit disk $D^{m_0}$ or a $1$-handlebody. $M_{0,j} \times {{D^{k^{\prime}}}_i}$ is represented as the product of a copy of the unit disk $D^{m_0}$ or a $1$-handlebody and a copy of the unit disk $D^{k^{\prime}}$. The intersection of this set with the boundary $\partial f(M)$ of the image $f(M)$ is regarded as the product of $M_{0,j}$ and a copy of the unit disk $D^{k^{\prime}-1}$ embedded smoothly in the boundary  $\partial {D^{k^{\prime}}}_i \subset {D^{k^{\prime}}}_i$ here.

	Here, note that 1-handlebodies are (simple) homotopy equivalent to bouquets of finitely many circles. Note also that in the case of the multisection, the structure group of the given linear bundle is assumed to be the rotation group.

    This makes the restriction of the internal smooth bundle of $f$ to the intersection of $M_{0,j} \times {D^{k^{\prime}}}_i$ and the base space to be a trivial smooth bundle.
    This argument, Proposition \ref{prop:2} and the local structures of special generic maps complete the proof.
	
\end{proof}
\begin{Ex}
	\label{ex:3}
	Furthermore, we can obtain maps in Example \ref{ex:1} as maps for Theorem \ref{thm:6} in the case $l=1$.
	Special generic maps in Theorem \ref{thm:1} (Theorem \ref{thm:1} (\ref{thm:1.3}) with $l=1$) are regarded as maps for Theorem \ref{thm:6} in considerable cases .

\end{Ex}
\begin{Ex}
\label{ex:4}
		Linear bundles whose fibers are circles are classified via the 2nd integral cohomology group of the base space. In general they give a family of infinitely many total spaces of bundles	distinct spaces in which are not mutually homeomorphic.
		  
			We give another example. Finite iterations of taking linear bundles whose fibers are unit disks can help us to obtain families of infinitely many manifolds distinct manifolds in which are not mutually diffeomorphic. \cite{choimasudasuh} studies so-called {\it Bott manifolds} and their classifications. A {\it Bott manifold} is a manifold obtained by considering a finite iteration of considering linear bundles whose fibers are the $2$-dimensional unit sphere $S^2$ of an explicit class starting from a one-point set.
			
	Such cases can give examples for Theorem \ref{thm:6}.

	\end{Ex}

\begin{Thm}[Main Theorem \ref{thm:2} and a related new result]
	\label{thm:7}
		Let $m>n \geq 1$ be integers. 
		\begin{enumerate}
		
			\item \label{thm:7.1}
			Let $m>n \geq 1$ be integers. Let $M_L$ be an {\rm (}$m-n+1${\rm )}-dimensional smooth, compact and connected manifold.
		Assume that there exists an admissible smooth {\rm (}resp. PL{\rm )} decomposition defined by a family $\{M_{L,j}\}_{j=1}^k$ of submanifolds. Then a manifold $M$ admitting a round fold map $f:M \rightarrow {\mathbb{R}}^n$ having a globally trivial monodromy and having an axis manifold diffeomorphic {\rm (}resp. PL homeomorphic{\rm )} to $M_L$ admits an admissible smooth {\rm (}PL{\rm )} decomposition of degree $2k$. In addition in the case where the given admissible decomposition defined by the family $\{M_{L,j}\}_{j=1}^k$ of the submanifolds is a disk decomposition, the resulting admissible decomposition of degree $2k$ is also a disk decomposition. In the case where the given admissible decomposition defined by the family $\{M_{L,j}\}_{j=1}^k$ of the submanifolds is a multisection, the resulting admissible decomposition of degree $2k$ is a multisection.
		In addition, the given admissible decomposition is DGNANM, then it is also DGNANM outside a closed manifold $f^{-1}(0)$.
			\item \label{thm:7.2}
			Let $M_L$ be an {\rm (}$m-n+1${\rm )}-dimensional closed and connected manifold.
	Consider a manifold $M$ admitting a round fold map $f:M \rightarrow {\mathbb{R}}^n$ having a globally trivial monodromy and an axis manifold diffeomorphic {\rm (}PL homeomorohic{\rm )} to a manifold $M_{L,{\rm (}l{\rm )}}$. Furthermore assume that the following properties are enjoyed.
	\begin{enumerate}
		\item There exists a smooth {\rm (}resp. PL{\rm )} DGNANM decomposition of $M_L$ defined by a family $\{M_{L,j}\}_{j=1}^k$ of submanifolds. Then the intersection ${\bigcap}_{j=1}^k M_{L,j}$ contains a {\rm (}discrete{\rm )} subset of exactly $l>0$ points.
		\item $M_{L,{\rm (}l{\rm )}}$ is a smooth manifold obtained by removing the interiors of exactly $l$ smoothly and disjointly embedded copies of the unit disk $D^{m-n+1}$ in $M_L$. 
	\end{enumerate}
	Then $M$ admits a smooth {\rm (}resp. PL{\rm )} DGNANM decomposition outside $f^{-1}(0)$ of degree $2k$. In addition in the case where the given DGNANM decomposition defined by the family $\{M_{L,j}\}_{j=1}^k$ is a disk decomposition, the admissible decomposition of degree $2k$ is a disk decomposition. In the case where the given DGNANM decomposition defined by the family $\{M_{L,j}\}_{j=1}^k$ is a multisection, the admissible decomposition of degree $2k$ is a multisection.
\end{enumerate}
\end{Thm}
	\begin{proof}[A proof of Main Theorem \ref{thm:2} or Theorem \ref{thm:7} {\rm (}\ref{thm:7.1}{\rm )} and Theorem \ref{thm:8}] 
		In the case $n=1$, we consider a family of submanifolds consisting of exactly $l$ manifolds each of which is diffeomorphic to the corresponding submanifold in the family $\{M_{L,j}\}_{j=1}^l$ and another $l$ manifolds where each of the latter $l$ manifolds is diffeomorphic to the corresponding submanifold in the family $\{M_{L,j}\}_{j=1}^l$ as before.
		 
		We concentrate on the case $n \geq 2$. For a round fold map into ${\mathbb{R}}^n$ here, we consider a natural construction, called {\it spinning construction} in \cite{kitazawa0.2} first. More precisely, we consider the product map of a Morse function on $M_L$ and the identity map on the unit sphere $S^{n-1}$. We embed the image smoothly and suitably into ${\mathbb{R}}^n$. We glue the projection of a trivial bundle over a copy of the unit disk $D^n$ in a suitable way by a suitable bundle isomorphism between the naturally obtained trivial smooth bundles. 
		
		$M_L \times D^n$ has a natural admissible decomposition defined by the family $\{M_{L,j} \times D^{n-1}\}_{j=1}^k$.
		In a situation like one in Definition \ref{def:10} (\ref{def:10.4}), we put $X:=M$, $Y:=M_L$, $F_{Y_i}=D^{n-1}$ and $E_Y=D^n$ for example and this completes the proof of Theorem \ref{thm:8}, presented later. 
		For example, we glue these products by the product of two diffeomorphisms (resp. piecewise smooth homeomorphisms) where the diffeomorphism (resp. piecewise smooth homeomorphism) between the base spaces is one for the identification to obtain the Euclidean space ${\mathbb{R}}^n$ of the target of our desired round fold map.
		However, in general, we cannot argue in such a way and admissible decompositions may not be DGNANM.
		
		As in FIGURE \ref{fig:1}, the map $f$ and the manifold of the domain are divided into two copies. The copy of the manifold is diffeomorphic to $M_L \times D^{n-1}$ and the copy of the map there is regarded as the product map of the Morse function on $M_L$ and the identity map on the unit disk $D^{n-1}$.
		
		\begin{figure}
			
			\includegraphics[width=30mm,height=30mm]{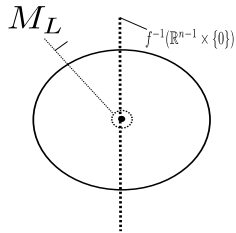}
			\caption{The image of a round fold map $f$ and the preimages of several submanifolds depicted by dotted straight segments): we identify as ${\mathbb{R}}^{n-1} \times \{0\} \subset {\mathbb{R}}^{n-1} \times \mathbb{R}={\mathbb{R}}^n$ in the canonical way for example.}
			\label{fig:1}
		\end{figure}
		
		We can complete the proof in the case where the decompositions may not be DGNANM.
		
		We can push the copies of the manifold and the copies of the (product) map into the complements of the central disk and the preimage by using diffeomorphisms or a smooth isotopy such that at each value of the space of the parameter, represented as the closed interval $[0,1]$, we have a smooth diffeomorphism. Due to this, we can complete the proof in the case where the given admissible decomposition is DGNANM.
		
		We can show statements on disk decompostions and multisections easily.
		
		This completes the proof.
		\end{proof}
	
	We present an example for Main Theorem \ref{thm:2}. This is not presented in the second section. This is not understood via Theorems \ref{thm:2}, \ref{thm:3} or \ref{thm:4} or Example \ref{ex:2}. ${\mathbb{C}P}^k$ denotes the {\it $k$-dimensional complex projective space}. This is a $k$-dimensional complex manifold and $2k$-dimensional closed and simply-connected manifold. For fundamental notions from algebraic topology such as (integral) cohomology rings and cup products of elements of them, see \cite{hatcher} for example. 
	\begin{Ex}
		\label{ex:5}
		According to \cite{kitazawa0.6}, $7$-dimensional smooth, spin and simply-connected manifolds whose integral cohomology rings are isomorphic to that of ${\mathbb{C}P}^2 \times S^3$ always admit  round fold maps into ${\mathbb{R}}^4$ having globally trivial monodromies and axis manifolds diffeomorphic to a manifold obtained by removing the interior of a copy of the unit disk $D^4$ smoothly embedded in the interior of $S^2 \times D^2$. This axis manifold admits a smooth GN disk decomposition of degree $2$.
		For this, first we consider $S^2 \times D^2$ as a manifold represented as one obtained by gluing two copies of $D^2 \times D^2$ by the product map of a diffeomorphism on $\partial D^2$ and the identity map on $D^2$. After that we can remove the interior of a small smoothly embedded copy of the $4$-diensional unit disk $D^4$ as in Proposition \ref{prop:4}. Note that for these operations we consider in the smooth category. We can see that the $7$-dimensional manifold admits a smooth DGNANM disk decomposition of degree $4$ outside $f^{-1}(0)$ where $f:M \rightarrow {\mathbb{R}}^n$ denotes a round fold map for our case. 
		As presented in the end of the third section, this is not {\it GN outside $f^{-1}(0)$} and this is {\it GANM outside $f^{-1}(0)$} where we can define such new notions as specific cases of DGNANM decompositions outside compact and connected submanifold with no boudanry.
		
		According to Wang's classification of such $7$-dimensional manifolds in \cite{wang}, all such manifolds admit such round fold maps. Moreover, without applying the classification, \cite{kitazawa0.6} obtains infinitely many maps and the family of the manifolds distinct manifolds in which are not mutually homeomorphic.
		
		According to fundamental methods in the theory of Lusternik-Schnirelmann cateogy numbers, for each manifold $M$ of such $7$-dimensional manifolds here, we have ${\rm cat}(M)=3$. We apply the fact that we have three elements whose degrees are positive and whose cup products are not the zero element in the integral cohomology ring. 
	\end{Ex}
	We can prove Theorem \ref{thm:7} {\rm (}\ref{thm:7.2} in a similar way and present important ingredients of the proof only.

	\begin{proof}[Important ingredients of a proof of Theorem \ref{thm:7} {\rm (}\ref{thm:7.2}{\rm )}]
		We can prove in a way similar to the proof of Main Theorem \ref{mthm:2}. For $M_{L,(l)}$ and its suitable admissible decomposition, we remove the interiors of $l$ copies of the unit disk $D^{\dim M_{L,(l)}}$ smoothly and disjointly embedded in $M_L$ and their given admissible decompositions respecting Proposition \ref{prop:4}. This is another main ingredient of the proof.
	\end{proof}
For example, Theorem \ref{thm:2} presents some simplest examples.
In the case $\Sigma$ is an exotic sphere there, we should discuss the problem in the PL category. 

We have the following by considering our proof of Theorem \ref{thm:7} and it is shown there.

\begin{Thm}
	\label{thm:8}
Let $M_L$ be an {\rm (}$m-n+1${\rm )}-dimensional smooth, compact and connected manifold. Assume that there exists a smooth {\rm (}resp. PL{\rm )} DGNANM decomposition of it defined by a family $\{M_{L,j}\}_{j=1}^k$ of submanifolds. Then there exists a manifold $M$ admitting a round fold map into ${\mathbb{R}}^n$ having a globally trivial monodromy and an axis manifold diffeomorphic {\rm (}resp. PL homeomorphic{\rm )} to $M_L$ and admitting a smooth {\rm (}resp. PL{\rm )} DGNANM decomposition of degree $2k$. In addition, in the case where the given DGNANM decomposition defined by the family $\{M_{L,j}\}_{j=1}^k$ is a disk decomposition, the resulting DGNANM decomposition of degree $2k$ is a disk decomposition. In the case where the given DGNANM decomposition defined by the family $\{M_{L,j}\}_{j=1}^k$ is a multisection, the resulting DGNANM decomposition of degree $2k$ is a multisection.
\end{Thm}

\section{Conclusion and future problems.}
Consider using our Main Theorems or Theorems \ref{thm:6} and \ref{thm:7}, Theorem \ref{thm:8} and connected sums as in Definition \ref{def:10} (\ref{def:10.5}) or Example \ref{ex:6}, which is presented here, for example, inductively. We have various closed and connected manifolds, special generic maps or round fold maps and their admissible decompositions (which are DGNANM outside some compact submanifolds with no boundaries) compatible with the fold maps in some senses.

As an example, we explain about construction via using connected sums.

\begin{Ex}
	\label{ex:6}
	Consider resulting special generic maps in Main Theorem \ref{mthm:1} or Theorem \ref{thm:6} where the admissible decompositions are DGNANM.
	
	We can consider a special generic map on a manifold represented as a connected sum of the given two manifolds whose image is a manifold represented as a boundary connected sum of the original two images into some connected non-closed manifold with no boundary. Note that here connected sums and boundary connected sums are taken in the smooth category.
	
	In some suitable specific case, for each special generic map. we can choose a copy of the unit disk which is regarded as a regular neighborhood of a point which is in the boundary and whose dimension is same as that of the image satisfying the following conditions.
	
	\begin{enumerate}
		\item These copies are also for Proposition \ref{prop:4}.
		 \item We can also do so that the preimages are also for Proposition \ref{prop:4}.
		 \item We obtain a situation as in Definition \ref{def:10} (\ref{def:10.5}).
		 \end{enumerate}
	 
	 For this, remember arguments on DGNANM decompositions of the images of special generic maps and other arguments in the proof of Theorem \ref{thm:6} for example.

	One of most trivial cases is given by considering two copies of a special generic map as in Theorem \ref{thm:6} where the admissible decompositions are DGNANM.
	
	We also have a DGNANM decomposition of the resulting manifold of the domain defined by a family of submanifolds each of which is the preimage of the preimage of a submanifold in the canonically obtained family defining a natural DGNANM decomposition of the manifold of the image. 	
\end{Ex}

	We close our paper by several problems. The following is regarded as a reminder about Problems \ref{prob:1} and \ref{prob:2}.
	\begin{Prob}
		Can we formulate more meaningful and better classes of decompositions (via generic smooth maps whose codimensions are negative)? Can we find nice examples for our classes of decompositions more? 
	\end{Prob}
The following are more explicit problems.
\begin{Prob}
	In some theorems and arguments in our paper, we consider admissible disk decompsoitions of copies of unit disks such that by considering the restrictions to the boundaries we have admissible disk decompositions of standard spheres. This is also a key ingredient in \cite{kitazawa8}.
		Can we find such admissible disk decompositions for general compact and connected manifolds explicitly? Can we apply such decompositions for our new construction of explicit admissible decompositions of explicit classes?
\end{Prob}
\begin{Prob}
	By applying our construction or the theory of Main Theorems or Theorems \ref{thm:6} and \ref{thm:7} and Theorem \ref{thm:8} for example as in the beginning of this section, construct and study explicit examples concretely and systematically. As an explicit strategy, it may be nice to start from well-known examples such as ones in Theorem \ref{thm:5}. 
\end{Prob}
\begin{Prob}
    Can we consider admissible decompositions for round fold maps which may not have globally trivial monodromies such as ones presented in Theorems \ref{thm:3} and \ref{thm:4} and Example \ref{ex:2}. Note that in \cite{kitazawasaeki1}, we encounter such round fold maps on $3$-dimensional closed, connected and orientable manifolds into ${\mathbb{R}}^2$.  
\end{Prob}

	\end{document}